\documentclass[a4paper,draft]{amsproc}
\usepackage{amssymb}
\usepackage{amscd} %% Package for commutative diagrams
\usepackage[lite]{amsrefs}
\usepackage{latexsym,enumerate}

 \usepackage[notref, notcite]{}
\usepackage{amsfonts, eucal,amsmath,amsthm}

\theoremstyle{plain}
 \newtheorem{thm}{Theorem}[section]

 \newtheorem{cor}{Corollary}[section]
\theoremstyle{definition}

\theoremstyle{remark}
 \newtheorem{rem}{Remark}[section]
 \numberwithin{equation}{section}

\renewcommand{\le}{\leqslant}\renewcommand{\leq}{\leqslant}
\renewcommand{\ge}{\geqslant}\renewcommand{\geq}{\geqslant}

\newcommand{\norm}[2]{\left\|#1\right\|_{#2}}
\newcommand{\andd}{\quad\mbox{\rm and}\quad}

\def\be  {\begin{equation}}
\def\ee  {\end{equation}}
\def \Lip{\mathop{\rm Lip}\nolimits}
\newcommand{\B}{\mathbb B}
\newcommand{\C}{\mathbb C}
\newcommand{\W}{\mathbb W}
\newcommand{\R}{\mathbb R}

\newcommand{\N}{\mathbb N}
\renewcommand{\a}{\alpha}
\renewcommand{\b}{\beta}
\newcommand{\ineq}[1]{(\ref{#1})}

\newcommand{\ie}{{\em i.e., }}

\newcommand{\st}{\;\; \big| \;\;}
\renewcommand{\L}{\mathbb{L}}
\newcommand{\Lp}{\L_p}
\newcommand{\tL}{\mathbb{\widetilde L}}
\newcommand{\tLp}{\tL_p}
 \newcommand{\tCon}{\widetilde \C}
\newcommand{\tE}{\widetilde E}
\newcommand{\tD}{\widetilde D}
\newcommand{\tI}{\widetilde I}
\newcommand{\tC}{\widetilde C}
\newcommand{\tW}{\widetilde \W}
\newcommand{\tLip}{\widetilde \Lip}

\newcommand{\wab}{w_{\a,\b}}
\newcommand{\ds}{\displaystyle}
  \newcommand{\AC}{\mathrm{AC}}
  \newcommand{\loc}{\mathrm{loc}}

 \newcommand{\ec}{\end{comment}}
\newcommand{\bc}{ \begin{comment} }

\newenvironment{comment}[2]
{\bgroup\vspace{7pt}
\begin{tabular}{|p{5in}|}
\hline \qquad \bf \footnotesize Comment -- to be deleted in the final version \\
\hline
\quad\sl\footnotesize #1#2} {\\ \hline \end{tabular}
\vspace{7pt}\indent\egroup}

\newtheorem{theoremdir}{Theorem}

\newtheorem{theoreminv}{Theorem}

\newtheorem{theoremcon}{Theorem}

\newtheorem{theoremdirr}{Theorem}

\newtheorem{theoreminvr}{Theorem}

\newtheorem{theoremconr}{Theorem}

\newtheorem{theoremdira}{Theorem}

\newtheorem{theoreminva}{Theorem}

\newtheorem{theoremcona}{Theorem}

\newtheorem{theoremdirar}{Theorem}

\newtheorem{theoreminvar}{Theorem}

\newtheorem{theoremconar}{Theorem}

\title[New moduli of smoothness]{NEW MODULI OF SMOOTHNESS}

\subjclass[1991]{Primary 41A17; Secondary 41A10, 42A10, 41A25, 41A27}

\author[Kopotun]{\bfseries K. A. Kopotun}
\address{
Department of Mathematics \\
University of Manitoba   \\
Winnipeg, Manitoba R3T 2N2\\
Canada }
\email{kopotunk@cc.umanitoba.ca}
\thanks{The first author acknowledges  support of NSERC of Canada.}

\author[Leviatan]{\bfseries D. Leviatan}
\address{
Raymond and Beverly Sackler School of Mathematical Sciences\\
 Tel Aviv University\\
 Tel Aviv 69978\\
 Israel }
\email{leviatan@post.tau.ac.il}

\author[Shevchuk]{\bfseries I. A. Shevchuk}
\address{
 Faculty of Mechanics and Mathematics\\
 Taras Shevchenko National University of Kyiv\\
  Kyiv 01601\\
 Ukraine}
\email{shevchuk@univ.kiev.ua}

\begin{document}

\vspace{18mm} \setcounter{page}{1} \thispagestyle{empty}

\begin{abstract}
In this paper, we discuss various properties of the new modulus of smoothness
\[
\omega^\varphi_{k,r}(f^{(r)},t)_p := \sup_{0 < h\leq t}\|\mathcal W^r_{kh}(\cdot)
\Delta_{h\varphi(\cdot)}^k (f^{(r)},\cdot)\|_{\Lp[-1,1]},
\]
where $\varphi(x):=\sqrt{1-x^2}$ and
$\ds
\mathcal W_\delta(x) = \bigl((1-x-\delta\varphi(x)/2)
(1+x-\delta\varphi(x)/2)\bigr)^{1/2}.
$

Related moduli with more general weights are also considered.
\end{abstract}

\maketitle

\section{Introduction}

\subsection{Trigonometric approximation}

Let $\tLp$, $1\le p\le\infty$, denote the space of $2\pi$-periodic measurable functions for which  the norm
$$
\|f\|_{\tLp}:=\left( \int_{-\pi}^\pi |f(x)|^p dx \right)^{1/p}
$$
is finite. Here, by $\tL_\infty$ we mean the space of continuous $2\pi$-periodic functions   $\tCon$ equipped with the uniform norm, \ie
$$
\|f\|_{\tCon}:=\max_{x\in[-\pi,\pi]}|f(x)|.
$$

Let $\mathcal T_n$, $n\in \mathbb{N}$, be the space of $(n-1)$st degree trigonometric
polynomials
$$
T_n(x)= \sum_{j=0}^{n-1}(a_j\cos jx+b_j\sin jx).
$$
For $f\in\tLp$, denote by
\be\label{delta}
\Delta_h^k(f,x)= \sum_{i=0}^k  {\binom ki} (-1)^{k-i} f(x+(i-k/2)h)
\ee
the $k$th symmetric difference of the function $f$, and by
$$
\omega_k(f,t)_p:=\inf_{h\in[0,t]}\|{\Delta}_h^k(f,\cdot)\|_{\tLp}
$$
its $k$th modulus of smoothness. Finally, let
$$
\tE_n(f)_p:=\inf_{T_n\in\mathcal T_n}\|f-T_n\|_{\tLp},
$$
denote the degree of approximation of $f$ by trigonometric polynomials from $\mathcal T_n$.

In 1908, de la Vall{\'e}e Poussin (see \cite{dlvp-1925}*{Section 7}, for example) posed a problem on a connection between the rate of polynomial approximation of functions and their differential properties. To quote
de la Vall{\'e}e Poussin \cite{dlvp-1925}*{p. 119}, ``It is the memoir by D. Jackson \cite{jack} which answers most completely the direct question, and that of S. Bernstein \cite{bern} which answers most completely the inverse problem''. These results were   generalized by  de la Vall{\'e}e Poussin  in \cite{dlvp-1919}, though as he writes in \cite{dlvp-1925}*{p. 119}, ``I combined the results obtained by the two authors above
named, and filled them out in many points; I changed or
simplified the proofs; but I contributed little in the way of
new materials to the construction''.

In 1911, D. Jackson \cite{jack}*{Theorem VIII} (see also \cite{jack-1921}*{p. 428})  proved the following inequality (which is now commonly known as one of ``Jackson's inequalities''):
$$
\tE_n(f)_\infty \le c\omega_1(f,n^{-1})_\infty , \quad n\ge 1.
$$
This result  was later   extended by Zygmund \cite{zyg}*{Theorems $8$ and $8'$}, % ($k=2$, $\omega_2(f, t)_p \sim t$),
Bernstein \cite{ber}, % ($\omega_k(f, t)_\infty \sim t^\alpha$, $p=\infty$),
Akhiezer \cite{ach}*{Section 89}, % ($k=1$, $1\leq p<\infty$ and $k=2$, $p=\infty$)
and Stechkin \cite{ste}*{Theorem 1} %($k \geq 1$, $p=\infty$ with a proof in the case $p<\infty$ being similar)
as follows.

\begin{theoremdir}[{\bf D}irect theorem, $r=0$]
 Let $k\in\N$. If $f\in\tLp$, $1\leq p \leq \infty$,  then
$$
\tE_n(f)_p\le c(k)\omega_k(f,n^{-1})_p, \quad n\ge 1.
$$
\end{theoremdir}

We note that    ``$r=0$''  and the subscript ``$0$'' in ``$\mathrm{\tD_0}$'' will become clear once one compares this result with Theorem~$\mathrm{\tD_r}$  below.

%A matching inverse theorem ({\em i.e.}, an inverse theorem in terms of the $k$-th modulus of smoothness) was proved by the  Timan brothers  \cite{timantiman,ti-diss} and  Stechkin \cite[Theorem 8]{ste} ($p=\infty$), see also Quade \cite[Theorem 1]{quade} ($k=1$), de la Vall{\'e}e Poussin  \cite[Section 39]{dlvp-1919} ($k=1$ and $p=\infty$) and Salem \cite[Chapter V]{salem} ($k=1$ and $p=\infty$).

Matching inverse theorems are due to Bernstein \cite{bern},
de la Vall{\'e}e Poussin  \cite{dlvp-1919}*{Section 39}, % ($k=1$ and $p=\infty$),
Quade \cite{quade}*{Theorem 1}, % ($k=1$),
Salem \cite{salem}*{Chapter V}, % ($k=1$ and $p=\infty$),
Zygmund \cite{zyg}*{Theorems $8$, $8'$, $9$ and $9'$},
the  Timan brothers  \cite{timantiman},
 and  Stechkin \cite{ste}*{Theorem 8}. % ($p=\infty$).

\begin{thm}
 Let $k\in\N$  and  $f\in\tLp$, $1\leq p\leq \infty$. Then
\[
\omega_k(f,n^{-1})_p \le\frac{c(k)}{n^k}\sum_{\nu=1}^n \nu^{k-1} \tE_{\nu}(f)_p , \quad n\geq 1.
\]
\end{thm}

This theorem can be restated in the following form.

\begin{theoreminv}[{\bf I}nverse theorem, $r=0$]
 Let $k\in\N$  and let
$\phi:[0,1]\mapsto[0,\infty)$ be a nondecreasing
function such that $\phi(0+)=0$.
If a function $f\in\tLp$, $1\leq p\leq \infty$, is such that
\[
\tE_n(f)_p\le \phi\left(n^{-1}\right), \quad n\ge 1,
\]
then
\[
\omega_k(f,t)_p\le c(k)t^k  \int_t^1\frac{\phi(u)}{u^{k+1}}du      ,\quad
0<t\le1/2.
\]
\end{theoreminv}

These direct and inverse theorems yield a constructive characterization of the class
$\tLip(\alpha,p) = \left\{ f\in\tLp \st \omega_{\lfloor \alpha \rfloor +1} (f,t)_p \leq c t^\alpha\right\}$.

 \begin{theoremcon}[{\bf C}onstructive characterization, $r=0$]
   Let $f\in\tLp$, $1\leq p \leq \infty$, and $\alpha > 0$. %  and $0<\alpha\le k$.
    If
$$
\omega_k(f,t)_p\le t^\alpha,
$$
then
$$
\tE_n(f)_p\le c(k) n^{-\alpha},\quad n\ge 1.
$$
Conversely, if $0<\alpha<k$ and
$$
\tE_n(f)_p\le  n^{-\alpha},\quad n\ge 1,
$$
then
$$
\omega_k(f,t)_p\le c(k,\alpha)t^\alpha.
$$
\end{theoremcon}

 Jackson's inequalities of the second type involve differentiable functions.

Let $\tW^r_p$, $r\in\mathbb N$, be the space of $2\pi$-periodic
functions $f$ such that $f^{(r-1)}$ is absolutely continuous  and $f^{(r)}\in\tLp$,
where by $\tW^r_\infty$ we mean $\tCon^r$.

The following result is an immediate consequence of Theorem~$\mathrm{\tD_0}$ and the well known property
\[
\omega_{k+r}(f ,t)_p \leq t^r  \omega_k(f^{(r)},t)_p, \quad t>0 .
\]

\begin{theoremdirr}[Direct theorem, $r\in\N$]
    Let $k\in\N$ and $r\in\N$. If $f\in\tW^r_p$, then
$$
\tE_n(f)_p\le c(k,r) n^{-r} \omega_k(f^{(r)},n^{-1})_p,\quad
n\ge 1.
$$
\end{theoremdirr}

%The following result was obtained by A. F. Timan \cite{ti-diss}, \cite[Theorem 6.1.3]{ti-book}, see also Quade \cite[Theorem 1]{quade} ($k=1$),  de la Vall{\'e}e Poussin  \cite[Section 39]{dlvp-1919} ($k=1$ and $p=\infty$)  and Stechkin \cite[Theorem 11]{ste} ($p=\infty$).

The following inverse theorems are due to
Bernstein \cite{bern}, % ($p=\infty$),
 de la Vall{\'e}e Poussin  \cite{dlvp-1919}*{Section 39}, % ($k=1$ and $p=\infty$),
 Quade \cite{quade}*{Theorem 1}, % ($k=1$),
  Zygmund \cite{zyg}*{Theorems $8$, $8'$, $9$ and $9'$},
 Stechkin \cite{ste}*{Theorem 11}, % ($p=\infty$)
 and
A. Timan \cite{ti-diss}, \cite{ti-book}*{Theorem 6.1.3}.

\begin{thm}
 Let   $r\in\N$  and  $f\in\tLp$, $1\leq p\leq \infty$. If
 \[
 \sum_{\nu=1}^\infty \nu^{r-1} \tE_\nu(f)_p < \infty ,
 \]
 then $f$ is a.e. identical with a function from $\tW^r_p$.
In addition, for any $k\in\N$,
\[
\omega_k(f^{(r)},n^{-1})_p \le\frac{c(k,r)}{n^k} \sum_{\nu=1}^n \nu^{k+r-1}\tE_{\nu}(f)_p  + c(k,r)  \sum_{\nu=n+1}^\infty \nu^{r-1} \tE_{\nu}(f)_p , \quad n\geq 1.
\]
\end{thm}

This theorem can be restated as follows.

\begin{theoreminvr}[Inverse theorem, $r\in\N$]
 Let $k\in\N$, $r\in\N$ and $\phi:[0,1]\mapsto[0,\infty)$ be a non
decreasing function such that $\phi(0+)=0$ and
$$
\int_0^1\frac{\phi(t)}{t^{r+1}}dt<\infty.
$$
If $f\in\tLp$ be such that
$$
\tE_n(f)_p\le\phi\left( n^{-1} \right), \quad n\ge 1,
$$
then $f$ is a.e. identical with a function from $\tW^r_p$,   and
\[
\omega_k(f^{(r)},t)_p\le c(k,r)\left(\int_0^t\frac{\phi(u)}{u^{r+1}}du+
t^k\int_t^1\frac{\phi(u)}{u^{k+r+1}}du\right),\quad
0<t\le1/2.
\]
\end{theoreminvr}

Finally, we have a constructive characterization of functions
$f\in\tW^r$ such that $f^{(r)}\in\tLip(\alpha-r,p)$.

\begin{theoremconr}[Constructive characterization, $r\in\N$]
 Let $r\in\N$, $\alpha > r$,
 $f\in\tW^r_p$, $1\leq p\leq \infty$, and
$$
\omega_k(f^{(r)},t)_p\le t^{\alpha-r}.
$$
Then
$$
\tE_n(f)_p\le c(k,r) n^{-\alpha},\quad n\ge 1.
$$
Conversely, if $f\in\tLp$, $1\leq p\leq \infty$, $r<\alpha<k+r$, and
$$
\tE_n(f)_p\le n^{-\alpha}, \quad n\ge 1,
$$
then $f$ is a.e. identical with a function from $\tW^r_p$ and
$$
\omega_k(f^{(r)},t)_p\le c(k,r,\alpha)t^{\alpha-r}.
$$
\end{theoremconr}

%\subsection{Direct and inverse results in trig. approximation: chronology}
%\mbox{}\\
%
%\begin{tabular}{ll}
%1911 &  Jackson \cite[Theorem VIII]{jack} (see also \cite[p. 428]{jack-1921}): Theorem~$\mathrm{\tD_0}$ for $k=1$ and $p=\infty$. \\
%1930 &  Jackson \cite[Theorem IV]{jack-1930}: Theorem~$\mathrm{\tD_r}$ for $k=1$ and $p=\infty$ (however, it seems that this result was known to Jackson already in 1911). \\
%1945&   Zygmund \cite{zyg} \\
%1947 &  Bernstein \cite{ber} \\
%1947 &   Achieser \cite[Section 89]{ach}: Theorem~$\mathrm{\tD_r}$ for $k=1$, $1\leq p<\infty$.\\
%
%1951 & Stechkin \cite[Theorem 1]{ste}: Theorem~$\mathrm{\tD_0}$ for $k \geq 1$ and  $p=\infty$. \\
%
%
%\end{tabular}

\subsection{Algebraic approximation}

Let $\Lp[-1,1]$, $1\le p\le\infty$ denote the usual $\Lp$ space equipped with the norm
\[
\|f\|_p:=\left(\int_{-1}^1|f(x)|^p dx\right)^{1/p},
\]
where by $L_\infty[-1,1]$ we mean $C[-1,1]$ equipped with the uniform norm.

Let $\mathcal P_n$ denote the space of algebraic polynomials of degree $<n$ and set
\[
E_n(f)_p:=\inf_{P_n\in\mathcal P_n}\|f-P_n\|_p
\]
the degree of best approximation of $f$ by algebraic polynomials in $\Lp$.

%Also, let $E_n(f)_{w,p}$ be the degree of   best weighted   approximation of $f$ by
%polynomials from $\mathcal P_n$. Namely,
%\[
%E_n(f)_{w,p}:=\inf_{P_n\in\mathcal P_n}\|(f-P_n)w\|_p.
%\]

Define
\[
\Delta^k_h(f,x;[-1,1]):=\begin{cases}\Delta^k_h(f,x),&\quad x\pm kh/2\in[-1,1],\\
0,&\quad{\rm otherwise},\end{cases}
\]
where $\Delta^k_h(f,x)$ was defined in \ineq{delta}.

Finally, define the Ditzian-Totik (DT) moduli of smoothness \cite{dt}, by
\be\label{omeg}
\omega_k^\varphi(f,t)_p:=\sup_{0<h\le t}\|\Delta^k_{h\varphi(\cdot)}(f,\cdot;[-1,1])\|_p
\ee
where $\varphi(x):=(1-x^2)^{1/2}$.

It is well known that the DT moduli of smoothness yield results which are   completely analogous to Theorems~$\mathrm{\tD_0}$, $\mathrm{\tI_0}$
and $\mathrm{\tC_0}$. Namely, we have the following results (see \cite{dt}).

\begin{theoremdira}
  Let $k\in\N$. If $f\in\Lp[-1,1]$, $1\le p\le\infty$,   then
$$
E_n(f)_p\le c(k)\omega^\varphi_k(f,n^{-1})_p,\quad n\ge k.
$$
\end{theoremdira}

\begin{theoreminva}
   Let $k\in\N$ and $\phi:[0,1]\mapsto[0,\infty)$ be a non
decreasing function such that $\phi(0+)=0$.
% and
%\[
%\int_0^1\frac{\phi(u)}{u^{k+1}}<\infty.
%\]
If a function $f\in \Lp[-1,1]$, $1\le p\le
\infty$, is such that
$$
E_n(f)_p\le \phi\left( n^{-1} \right), \quad n\ge k,
$$
then
$$
\omega^\varphi_k(f,t)_p\le c(k)t^k\int_t^1\frac{\phi(u)}{u^{k+1}}du,\quad
t\in[0,1/2].
$$
\end{theoreminva}

\begin{theoremcona}
 Let $\alpha > 0$ and $f\in \Lp[-1,1]$, $1\le p\le \infty$.
%and $0<\alpha\le k$.
If
$$
\omega^\varphi_k(f,t)_p\le t^\alpha,
$$
then
$$
E_n(f)_p\le c(k) n^{-\alpha}\quad n\ge k.
$$
Conversely, if $0<\alpha<k$ and
$$
E_n(f)_p\le  n^{-\alpha}, \quad n\ge k,
$$
then
$$
\omega^\varphi_k(f,t)_p\le c(k,\alpha)t^\alpha.
$$
\end{theoremcona}

The purpose of this paper is to discuss our new moduli of smoothness (introduced in \cite{kls}) that
allow to obtain the analogs of Theorems~$\mathrm{\tD_r}$, $\mathrm{\tI_r}$ and $\mathrm{\tC_r}$.

\section{New moduli of smoothness}

\subsection{Definitions}

For $1\le p<\infty$ and $r\in\N$, denote
$$
\B_p^r:=\{f:f^{(r-1)}\in
AC_{loc}(-1,1)\quad\text{and}\quad\|f^{(r)}\varphi^r\|_p<+\infty\}.
$$

If $p=\infty$, then
$$
\B_\infty^r:=\{f:f\in C^r(-1,1)\quad\text{and}\quad\lim_{x\to\pm1}f^{(r)}(x)\varphi^r(x)=0\}.
$$
Finally, if $r=0$, then $\B_p^0:=\Lp[-1,1]$, $1\le p<\infty$ and $\B^0_\infty:=\C[-1,1]$.

For $f\in \B^r_p$, define
\[
\omega^\varphi_{k,r}(f^{(r)},t)_p := \sup_{0 < h\leq t}\|\mathcal W^r_{kh}(\cdot)
\Delta_{h\varphi(\cdot)}^k (f^{(r)},\cdot)\|_p,
\]
where
\[
\mathcal W_\delta(x):=\begin{cases}\bigl((1-x-\delta\varphi(x)/2)
(1+x-\delta\varphi(x)/2)\bigr)^{1/2},\\\qquad\qquad\qquad\quad1\pm x-\delta\varphi(x)/2
\in[-1,1],\\0,\qquad\qquad\qquad\,{\rm otherwise}.
\end{cases}
\]
Note that, if $r=0$, then
$$
\omega^\varphi_{k,0}(f,t)_p=\omega^\varphi_k(f,t)_p
$$
are the usual DT moduli defined in \ineq{omeg}.

It turns out (see \cite{kls}*{Lemma 3.2}) that if $f\in\B^r_p$, then
$$
\lim_{t\to 0+} \omega_{k,r}^\varphi(f^{(r)},t)_p = 0 .
$$
%In fact, for $p=\infty$, it is an if and only if relation.

\subsection{Weighted DT moduli of smoothness}

Let
\[\overrightarrow\Delta_h^kf(x):=\left\{
\begin{array}{ll}\ds
\sum_{i=0}^k  {\binom ki}
(-1)^{k-i} f(x+ih),&\mbox{\rm if }\, x,x+kh  \in[-1,1]\,,\\
0,&\mbox{\rm otherwise},
\end{array}\right.\]
and
\[\overleftarrow\Delta_h^kf(x):=\left\{
\begin{array}{ll}\ds
\sum_{i=0}^k  {\binom ki}
(-1)^{i} f(x-ih),&\mbox{\rm if }\,  x-kh,x  \in[-1,1]\,,\\
0,&\mbox{\rm otherwise},
\end{array}\right.\]
be the forward and backward $k$th differences, respectively. Note that
\[
  \overrightarrow\Delta_h^k f(x):= \Delta_h^k(f,x+kh/2) \andd
\overleftarrow\Delta_h^k f(x):= \Delta_h^k(f,x-kh/2) .
\]

Let
\be \label{wght}
w(x):=w_{\a,\b}(x) := (1-x)^\a (1+x)^\b ,
\ee
 where $\a,\b \geq 0$,
and denote
\[
\Lp(w) := \Lp(\wab) := \left \{f:[-1,1]\mapsto\R   \st  \|\wab f\|_p < \infty \right\}.
\]

For $f\in \Lp(w)$, the weighted DT moduli of smoothness were defined (see \cite{dt}*{(8.2.10) and Appendix B}) by
\begin{eqnarray} \label{mod}
  \omega^\varphi_k(f,t)_{w,p} & := & \sup_{0<h\le t}\|w\Delta_{h\varphi}^kf\|_{\Lp[-1+2k^2h^2 ,1-2k^2h^2]}  \\ \nonumber
& &+  \sup_{0<h\le 2k^2t^2 }
\|w \overrightarrow\Delta_h^kf\|_{\Lp[-1,-1+2k^2t^2]} \\ \nonumber
&& +\sup_{0<h\le 2k^2t^2}\|w \overleftarrow\Delta_h^kf\|_{\Lp[1-2k^2t^2,1]}.
\end{eqnarray}
 %
 %\begin{align}
%\omega^\varphi_k(f,t)_{w,p}&:=\sup_{0<h\le t}\|w\Delta_{h\varphi}^kf\|_{L_p[-1+t^*,1-t^*]}\\&+\sup_{0<h\le t^*}
%\|w\overrightarrow\Delta_h^kf\|_{L_p[-1,-1+12t^*]}+\sup_{0<h\le t^*}\|w\overleftarrow\Delta_h^kf\|_{L[1-12t^8,1]},\nonumber
%\end{align}
%where $t^*=k^2t^2$.
%
The first term on the right hand side of \ineq{mod}  is the main part modulus which is denoted by $\Omega^\varphi_k(f,t)_{w,p}$ (see \cite{dt}*{(8.1.2)}) and is further discussed in Section~\ref{DTlast}.

%It was shown in \cite[Theorem 6.1.4]{dt} that for $f\in\Lp(w) $ and $k\ge1$, we have
%$\omega^\varphi_{k+1}(f,t)_{w,p}\le M\omega^\varphi_k(f,t)_{w,p}$.

It was shown in \cite{dt}*{Theorem 6.1.1} that
$\omega^{\varphi}_k(f,t)_{w,p}$ is equivalent to the following weighted $K$-functional $K_{k,\varphi}(f,t^k)_{w,p}$ (with $0<t\leq t_0$):
\[
K_{k,\varphi}(f,t^k)_{w,p}:=\inf_{g^{(k-1)}\in\AC_{\loc}}\left(
\|(f-g)w\|_{p} +t^k\|w\varphi^kg^{(k)}\|_{p}\right).
\]

%Let the weight $w$ be such that $w\sim1$ in compacta in $(-1,1)$ and
%$w(x)\sim(1\mp x)^{\gamma(\pm1)}$, as $x\to\pm1$, where $\gamma(\pm1)\ge0$.
%
%Denote
%\[
%\Lp[-1,1]_w:=\{f:[-1,1]\mapsto\R\mid\|wf\|_p<\infty\},\quad1\le p\le\infty.
%\]
%For $f\in \Lp[-1,1]_w$, the weighted DT moduli of smoothness were defined (see \cite[(8.2.10) and Appendix B]{dt}) by
%\begin{align}\label{mod}
%\omega^\varphi_k(f,t)_{w,p}&:=\sup_{0<h\le t}\|w\Delta_{h\varphi}^kf\|_{\Lp[-1+t^*,1-t^*]}\\&+\sup_{0<h\le t^*}
%\|w\overrightarrow\Delta_h^kf\|_{\Lp[-1,-1+12t^*]}+\sup_{0<h\le t^*}\|w\overleftarrow\Delta_h^kf\|_{L[1-12t^*,1]},\nonumber
%\end{align}
%where $t^*=2k^2t^2$.
%
%Given $k\in\mathbb N$ and weight $w$ as above, the following $K$-functional was defined in \cite[Section 6.1]{dt}
%$$
%K^\varphi_k(f,t^k)_{w,p}:=\inf(\|(f-g)w\|_p
%+t^k\|g^{(k)}w\varphi^k\|_p),
%$$
%where the infimum is taken over all $g\in \Lp[-1,1]_w$ such that $g^{(k)}\in \Lp[-1,1]_{w\varphi^k}$.
%
%It was proved in \cite[Theorem 6.1.1]{dt} that
%\be\label{equivalineq}
%\omega^\varphi_k(f,t)_{w,p}\sim K^\varphi_k(f,t^k)_{w,p}.
%\ee

\subsection{Properties of the new moduli}

For $r\ge0$ and $f\in\B_p^r$, we denote
$$
K^\varphi_{k,r}(f^{(r)},t^k)_p:=\inf_{g\in\B_p^{k+r}}(\|(f^{(r)}-g^{(r)})\varphi^r\|_p
+t^k\|g^{(k+r)}\varphi^{k+r}\|_p) .
$$
Then, we have the following equivalence results (see \cite{kls}*{Theorem 2.7}).

\begin{thm}  If $k\in\N$, $r\in\N_0$, $1\leq p\leq \infty$ and $f\in\B_p^r$,  then, for all $0<t\leq 2/k$,
\[
c  K^\varphi_{k,r}(f^{(r)},t^k)_p    \leq \omega_{k,r}^\varphi(f^{(r)},t)_p\leq c  K^\varphi_{k,r}(f^{(r)},t^k)_p   ,
\]
where constants $c$ may depend only on $k$, $r$ and $p$.
\end{thm}

\begin{cor}  If $k\in\N$, $r\in\N_0$, $1\leq p\leq \infty$ and $f\in\B_p^r$,  then,  for all $0<t\leq 2/k$,
\[
c K_{k,\varphi}(f^{(r)},t^k)_{\varphi^r,p}    \leq \omega_{k,r}^\varphi(f^{(r)},t)_p\leq c  K_{k,\varphi}(f^{(r)},t^k)_{\varphi^r,p}  .
\]
\end{cor}

Also, the following was proved in \cite{kls}*{Theorem 7.1}.
\begin{thm}\label{hierarchy}
If $f\in\B^{r+1}_p$, $1\le p \le \infty$, $r\in\N_0$ and $k\ge2$, then
\[
\omega_{k,r}^\varphi(f^{(r)},t)_p \leq
ct\omega_{k-1,r+1}^\varphi(f^{(r+1)},t)_p .
\]
\end{thm}

The following sharp Marchaud inequality was proved in \cite{dd}.

\begin{thm}[\mbox{\cite{dd}*{Theorem 7.5}}]
For $\a> -1/p$, $\b > -1/p$, $1<p<\infty$, $m \in\N$ and a weight $w$ defined in {\rm \ineq{wght}}, we have
\[
K_{m,\varphi}(f,t^m)_{w,p} \leq C t^m \left( \int_t^1 \frac{K_{m+1,\varphi}(f,u^{m+1})_{w,p}^q}{u^{mq+1}}\, du + E_m(f)_{w,p}^q \right)^{1/q}
\]
and
\[
K_{m,\varphi}(f,t^m)_{w,p} \leq C t^m \left( \sum_{n <1/t} n^{qm-1} E_n(f)_{w,p}^q \right)^{1/q} ,
\]
where $q = \min(2,p)$ and % $E_n(f)_{w,p} = \inf\{ \norm{ w(f-P)}{p} \st \deg P < n \}$.
 $E_n(f)_{w,p}$ is the degree of   best weighted   approximation of $f$ by
polynomials from $\mathcal P_n$, \ie
$E_n(f)_{w,p}:=\inf\left\{\|(f-P_n)w\|_p \st  P_n\in\mathcal P_n \right\}$.
\end{thm}

\begin{cor}
For   $1<p<\infty$, $r\in\N_0$, $m \in\N$ and $f\in\B^r_p$, we have
\[
\omega^{\varphi}_{m,r} (f^{(r)},t)_{p} \leq C t^m \left( \int_t^1\frac{\omega^{\varphi}_{m+1,r}(f^{(r)},u)_{p}^q z}{u^{mq+1}}\, du + E_m(f^{(r)})_{\varphi^r,p}^q \right)^{1/q}
\]
and
\[
\omega^{\varphi}_{m,r}(f^{(r)},t)_{p} \leq C t^m \left( \sum_{n <1/t} n^{qm-1} E_n(f^{(r)})_{\varphi^r,p}^q \right)^{1/q} ,
\]
where $q = \min(2,p)$.
\end{cor}

The following sharp Jackson inequality was proved in \cite{ddt}.

\begin{thm}[\mbox{\cite{ddt}*{Theorem 6.2}}]
For $\a> -1/p$, $\b > -1/p$, $1<p<\infty$, $m \in\N$ and a weight $w$ defined in {\rm \ineq{wght}}, we have
\[
2^{-nm} \left( \sum_{j=j_0}^n 2^{mjs} E_{2^j}(f)^s_{w,p} \right)^{1/s} \leq C K_{m,\varphi}(f, 2^{-nm})_{w,p}
\]
and
\[
2^{-nm} \left( \sum_{j=j_0}^n 2^{mjs}  K_{m+1,\varphi}(f, 2^{-j(m+1)})_{w,p}^s    \right)^{1/s} \leq C K_{m,\varphi}(f, 2^{-nm})_{w,p} ,
\]
where $2^{j_0} \geq m$ and $s = \max(p,2)$.
\end{thm}

\begin{cor}
For   $1<p<\infty$, $r\in\N_0$, $m \in\N$ and $f\in\B^r_p$, we have
\[
2^{-nm} \left( \sum_{j=j_0}^n 2^{mjs} E_{2^j}(f^{(r)})^s_{\varphi^r,p} \right)^{1/s} \leq C  \omega^{\varphi}_{m,r}(f^{(r)},2^{-n})_{p}
\]
and
\[
2^{-nm} \left( \sum_{j=j_0}^n 2^{mjs}  \omega^{\varphi}_{m+1,r}(f^{(r)},2^{-j})_{p}^s  \right)^{1/s} \leq C \omega^{\varphi}_{m,r}(f^{(r)},2^{-n})_{p},
\]
where $2^{j_0} \geq m$ and $s = \max(p,2)$.
\end{cor}

\begin{cor}
For   $1<p<\infty$, $r\in\N_0$, $m \in\N$ and $f\in\B^r_p$, we have
\[
t^m  \left(  \int_t^{1/m}        \frac{\omega^{\varphi}_{m+1,r}(f^{(r)},u)_{p}^s}{u^{ms+1}} \, du \right)^{1/s} \leq C \omega^{\varphi}_{m,r}(f^{(r)},t)_{p}, \quad 0<t\leq 1/m,
\]
where  $s = \max(p,2)$.
\end{cor}

%%%%%%%%%%%%%%%%%%%%%%%%%%%%%%%%%%%%%%%%%%%%%%%%%%%%%%%%%%%%%
\section{Algebraic polynomial approximation in $\Lp$}

In \cite{kls}, we   proved the following results analogous to Theorems $\mathrm{\tD_r}$, $\mathrm{\tI_r}$ and
$\mathrm{\tC_r}$ (see also \cite{kls-2010}*{Theorem 3.2} for the inverse result for $p=\infty$).

%First is the analog of Theorem $\tD_r$.

\begin{theoremdirar} If $f\in\B^r_p$, $1\le p\le\infty$, then
\be \label{dirarineq}
E_n(f)_p\le c(k,r) n^{-r}\omega_{k,r}^\varphi(f^{(r)},n^{-1})_p,\quad n\ge k+r.
\ee
\end{theoremdirar}

Note that it follows from the DT estimates that if $f\in\B^r_p$, then
$$
E_n(f)_p\le c(r) n^{-r}\|f^{(r)}\varphi^r\|_p,\quad n\ge r,
$$
which is asymptotically weaker than \ineq{dirarineq}.

It is also known that   if, for some $r\ge1$,
$f^{(r)}\in \Lp[-1,1]$, $1\le p\le\infty$,  then
$$
E_n(f)_p\le c(k,r) n^{-r}\omega^\varphi_k(f^{(r)},n^{-1})_p,\quad n\ge k+r.
$$
But we should emphasize that here we have to assume that $f^{(r)}\in \Lp[-1,1]$, as the DT-moduli are not
well defined if the function is not in $\Lp[-1,1]$ and, clearly, $\omega_{k,r}^\varphi(f^{(r)},n^{-1})_p$ is smaller than $\omega^\varphi_k(f^{(r)},n^{-1})_p$.

\begin{theoreminvar}  Let $r\in\N_0$, $k\ge1$, and $N\in\N$, and let
$\phi:[0,1]\mapsto[0,\infty)$ be a nondecreasing function such that $\phi(0+)=0$ and
\[
\int_0^1r\frac{\phi(u)}{u^{r+1}}\,du<\infty.
\]
If $f\in \Lp[-1,1]$, $1\le p\le\infty$, and
\[
E_n(f)_p\le\phi\left( n^{-1} \right),\quad \text{for all} \quad n\ge N,
\]
then $f$ is a.e. identical with a function from $\B^r_p$, and
\begin{eqnarray*}
\omega^\varphi_{k,r}(f^{(r)},t)_p &\le& c(k,r)\int_0^tr\frac{\phi(u)}{u^{r+1}}\,du+c(k,r) t^k\int_t^1\frac{\phi(u)}{u^{k+r+1}}\,du\\
&& +c(N,k,r)t^kE_{k+r}(f)_p, \qquad t\in[0,1/2].
\end{eqnarray*}
If, in addition, $N\le k+r$, then
\[
\omega^\varphi_{k,r}(f^{(r)},t)_p\le c(k,r)\int_0^tr\frac{\phi(u)}{u^{r+1}}\,du+c(k,r)t^k\int_t^1\frac{\phi(u)}{u^{k+r+1}}\,du,\quad t\in[0,1/2].
\]
\end{theoreminvar}

Taking $N =1$ and appropriately choosing the function $\phi$ we get the following corollary of Theorem~$\mathrm{I_r}$ in terms of the  degrees of approximation.

\begin{cor}
Given $1\leq p <  \infty$, $k\in\N$, $r\in\N_0$.
If
$$
\sum_{n=1}^\infty  r n^{r-1} E_n(f)_p <+\infty ,
$$
then $f$ is a.e. identical with a function from $\B^r_p$, and
\[
\omega_{k,r}^\varphi(f^{(r)},t)_p  \le  c\sum_{n > 1/t }r n^{r-1}E_n(f)_p    +ct^{k}\sum_{ 1\leq n\leq 1/t }n^{k+r-1}E_n(f)_p,\quad t\in\left[0,1/2\right].
\]
\end{cor}

\begin{theoremconar}
Let $r\in\N_0$, $\alpha > r$,  $k\ge1$ and $f\in\B^r_p$, $1\le p\le\infty$. If    % $r<\alpha\le r+k$ and
\[
\omega_{k,r}^\varphi(f^{(r)},t)_p\le  t^{\alpha-r},
\]
then
\[
E_n(f)_p\le c n^{-\alpha},\quad n\ge k+r.
\]
Conversely, if $r<\alpha<r+k$ and $f\in\Lp[-1,1]$ and
$$
E_n(f)_p\le n^{-\alpha},\quad n\ge N,
$$
then $f$ is a.e. identical with a function from $\B^r_p$, and
$$
\omega_{k,r}^\varphi(f^{(r)},t)_p\le
c(\alpha,k,r)t^{\alpha-r}+c(N,k,r)t^kE_{k+r}(f)_p,\quad t\in[0,1/2].
$$
If, in addition, $N\le k+r$, then
$$
\omega_{k,r}^\varphi(f^{(r)},t)_p\le
c(\alpha,k,r)t^{\alpha-r}.
$$
\end{theoremconar}

\section{Further characterizations}

In addition to characterizations in the previous section, we can also characterize certain smoothness classes of functions via the growth of certain weighted norms of their polynomials of best approximation.

%Another constructive characterization may be had by means of estimates on the growth of the norms of the polynomials of best approximation.

\begin{thm}
Let $f\in \Lp[-1,1]$, $1\le p\le\infty$, $k\in\N$, $r\in\N_0$,  $r<\alpha<r+k$, and suppose that  $P_n$ denotes the $(n-1)$st degree polynomial  of best approximation of $f$ in $\Lp[-1,1]$.  Then
\be\label{pngrowth}
\|\varphi^{r+k}P_n^{(r+k)}\|_p\le cn^{r+k-\alpha},\quad n\ge r+k,
\ee
if and only if  $f$ is a.e. identical with a function from $\B^r_p$, and
\be\label{omega}
\omega_{k,r}^\varphi(f^{(r)},t)_p\le c t^{\alpha-r},\quad t>0.
\ee
\end{thm}

\begin{proof} By virtue of \cite{dt}*{Theorem 7.3.1} we conclude that, for every $k\in\N$ and $r\in\N_0$,
\[
\|\varphi^{r+k}P_n^{(r+k)}\|_p\le c n^{k+r} \omega^\varphi_{k+r}(f,n^{-1})_p.
\]
Hence, if $f\in\B^r_p$ and \ineq{omega} is valid, then \ineq{pngrowth} follows immediately from the inequality
\be \label{auxjan}
\omega^\varphi_{k+r}(f,t)_p \leq c t^r \omega_{k,r}^\varphi(f^{(r)},t)_p
\ee
which is an immediate consequence of  Theorem~\ref{hierarchy}.

Conversely, if
\ineq{pngrowth} holds, then it follows by \cite{dt}*{Theorem 7.3.2} that $E_n(f)_p\le cn^{-\alpha}$, $n\ge r+k$.
Hence \ineq{omega} follows from Theorem~$\mathrm{C_r}$.
\end{proof}

%\bc
%
%In \cite[p. 72]{dt} there is a discussion on characterizing
%$\omega_r^\varphi(F,t)_p=O(t^\beta)$ with $r\ge2$ and
%$0<r-\beta\le1$. We may easily derive the conclusions that appear
%there from the original Jackson type estimates involving the DT
%moduli of smoothness and Theorem \ref{equiv}. Indeed, if
%\[
%\omega_r^\varphi(F,t)_p=O(t^\beta),\quad t\to0,
%\]
%Then, by the Jackson-type estimates, we have
%\[
%E_n(F)_p\le c\omega_r^\varphi(F,n^{-1})_p\le cn^{-\beta},\quad n\ge
%r=r-2+2.
%\]
%If $0\le r-2<\beta<r$ (the above mentioned constraints on $\beta$
%and $r$ (see \cite[p. 72, l. 9]{dt}) clearly have a typo in the
%right hand side inequality), then we apply \cite[Corollary 9.4]{kls}, with
%$k=2$ and $r$ replaced by $r-2$, and obtain that $F^{(r-2)}$
%exists a.e. in $(-1,1)$, and
%\[
%\omega^\varphi_{2,r-2}(F^{(r-2)},t)_p=O(t^{\beta-(r-2)}),
%\]
%or in \cite{dt}'s notation
%\[
%\omega_2^\varphi(F^{(r-2)},t)_{\varphi^{r-2},p}=O(t^{\beta-(r-2)}).
%\]
%In fact, we may readily extend the result in the following way.
%\begin{theorem}\label{DT}
%Let $k\in\N$ and $0\le r-k<\beta<r$, and assume $F\in \Lp[-1,1]$.
%If
%\[
%\omega_r^\varphi(F,t)_p=O(t^\beta),\quad t>0,
%\]
%then  $F^{(r-k)}$ exists a.e. in $(-1,1)$, and
%\[
%\omega^\varphi_{k,r-k}(F^{(r-k)},t)_p=O(t^{\beta-(r-k)}),\quad t>0.
%\]
%\end{theorem}
% \ec

We note that while the inequality \ineq{auxjan} cannot be reversed for a general function $f$, the following is an immediate consequence of Theorem~$\mathrm{C_r}$.

\begin{cor} \label{corequivjan}
Let $r\in\N_0$, $k\geq 1$, $f\in\Lp[-1,1]$, $1\leq p\leq \infty$,  $r<\alpha<r+k$.
If
\[
\omega_{r+k}^\varphi(f,t)_p \leq c t^\alpha ,\quad t>0 ,
\]
then $f$ is a.e. identical with a function from $\B^r_p$, and
\[
\omega^\varphi_{k,r}(f^{(r)},t)_p \leq c  t^{\alpha-r} ,\quad t>0.
\]
\end{cor}

\section{Further results for Weighted DT moduli} \label{DTlast}

The proofs (and therefore the results) of \cite{kls} may
be extended to the weighted DT moduli with weight $w$ which
satisfies the conditions of \cite{dt}*{Section 6.1}. So, in
particular, we have the hierarchy relations between the weighted
moduli of the function (of course, provided its derivative exists),
extending Theorem~\ref{hierarchy}.

\begin{thm}  Let $0<r<k$, and assume that $f$ is such
that $f^{(r-1)}$ is locally absolutely continuous in $(-1,1)$ and
$w\varphi^r f^{(r)}\in \Lp[-1,1]$, $1\le p\le\infty$. Then
\be\label{derivative} \omega_k^\varphi(f,t)_{w,p}\le
ct^r\omega_{k-r}^\varphi(f^{(r)},t)_{w\varphi^r,p},\quad t>0. \ee
\end{thm}

\begin{rem} The inequality \eqref{derivative} extends
\cite{dt}*{Corollary 6.3.3(b)}, as we do not require the condition of
$\beta(c)\ge1$, for $c=\pm1$, that appears there.
\end{rem}

\begin{proof}
Recall that the main part modulus $\Omega_k^\varphi$ is defined in \cite{dt}*{(8.1.2)} as follows:
\[
\Omega_k^\varphi(f,t)_{w,p}:=\sup_{0<h\leq t}
\norm{w\Delta_{h\varphi}^kf}{\Lp[-1+2k^2h^2 , -1+2k^2h^2]}.
\]
Then, \cite{dt}*{(6.2.9)} implies that
\[
\omega^\varphi_k(f ,t)_{w,p}\le
c\int_0^t(\Omega_k^\varphi(f ,\tau)_{w,p}\,/\tau)\,d\tau.
\]
Also, by \cite{dt}*{(6.3.2)}, we have
\[
\Omega_k^\varphi(f ,t)_{w,p}\le
ct\Omega_{k-1}^\varphi(f',t)_{w\varphi,p}.
\]
Hence,
\begin{align*}
\omega^\varphi_k(f ,t)_{w,p}
&\le
c\int_0^t\Omega_{k-1}^\varphi(f',\tau)_{w\varphi,p}\,d\tau\\
&\le
c t\Omega_{k-1}^\varphi(f',t)_{w\varphi,p}\le
ct \omega_{k-1}^\varphi(f',t)_{w\varphi,p},
\end{align*}
where for the second inequality we used the monotonicity of
$\Omega_{k-1}^\varphi(f',t)_{w\varphi,p}\,$, and for the
third we applied \cite{dt}*{(6.2.9)}.
Applying this inequality $r$ times we get the desired estimate.
\end{proof}

For the Jacobi weights $w=w_{\a,\b}$ defined in \ineq{wght},
it was proved by Ky \cite{ky}*{Theorem 4} (see also Luther and Russo \cite{lr}*{Corollary 2.2}) that there is an $n_0\in\N$ such that
\be \label{luther}
E_n(f)_{w,p}\le c\omega_k^\varphi(f,n^{-1})_{w,p},\quad n\ge n_0.
\ee

Thus, by \eqref{derivative}, we have the following Jackson-type result.

\begin{thm} Let $0<r<k$ and assume that $f^{(r-1)}$ is locally absolutely
continuous in $(-1,1)$ and $w\varphi^r f^{(r)}\in \Lp[-1,1]$, $1\le p\le\infty$. Then
\[
E_n(f)_{w,p}\le
cn^{-r}\omega_{k-r}^\varphi(f^{(r)},n^{-1})_{w\varphi^r,p},\quad n\ge n_0.
\]
\end{thm}

It was proved in \cite{dt}*{Theorem 8.2.4} that
\[
\omega_k^\varphi(f,t)_{w,p}\le ct^k\sum_{0<n\le1/t}n^{k-1}E_n(f)_{w,p},\quad t\le t_0 .
\]
This readily implies that, if $0<\alpha < k$ and $E_n(f)_{w,p}\le  n^{-\alpha}$, for $n\ge1$, then
\[
\omega_k^\varphi(f,t)_{w,p}\le c t^\alpha,\quad t\le t_0.
\]

In fact, it is possible to prove the following more general result.

\begin{thm}\label{Inverse} Let $0 \leq r<\alpha <k$, and let $f$ be such that $wf\in\Lp[-1,1]$, $1\le p\le\infty$. If, for some some $N\in\N$,
\be\label{Aleph}
E_n(f)_{w,p}\le n^{-\alpha}, \quad n\ge N,
\ee
then $f$ is a.e. identical with a function  that has a locally absolutely continuous derivative $f^{(r-1)}$ in $(-1,1)$, and
\[
\omega_{k-r}^\varphi(f^{(r)},t)_{w\varphi^r,p}\le
c(w,\alpha,k, r)t^{\alpha-r}+c(w,N,k,r)t^{k-r} E_{k}(f)_{w,p},\quad t>0.
\]
In particular, if $N\leq k$, then
\[
\omega_{k-r}^\varphi(f^{(r)},t)_{w\varphi^r,p}\le
c(w,\alpha,k,r)t^{\alpha-r},\quad t>0.
\]
\end{thm}

%\begin{cor} Let $r\in\N_0$, and $0\leq r < \alpha$, and suppose that $wf\in \Lp[-1,1]$. If, for some some $N\in\N$,
%\be\label{Aleph} E_n(f)_{w,p}\le n^{-\alpha}, \quad n\ge N, \ee
% then
%\[
%\omega_k^\varphi(f^{(r)},t)_{w\varphi^r,p}\le
%c(w,\alpha,k, r)t^{\alpha-r}+c(w,N,k,r)t^k E_{k+r}(f)_{w,p},\quad t>0.
%\]
%Moreover, if $N\leq k+r$, then
%\[
%\omega_k^\varphi(f^{(r)},t)_{w\varphi^r,p}\le
%c(w,\alpha,k,r)t^{\alpha-r},\quad t>0.
%\]
%\end{cor}

\begin{proof} Let $P_{k}\in\mathcal P_{k}$ be a polynomial of  best approximation to $f$ in the
weighted norm $\|w\cdot\|_p$, and set $F:=f-P_{k}$. Then
  $E_n(F)_{w,p}= \norm{wF}{p} = E_{k}(f)_{w,p}$, $n < k$, and
$E_n(F)_{w,p}=E_n(f)_{w,p}$, $n \geq k$. Hence, in particular, $E_n(F)_{w,p} \leq E_{k}(f)_{w,p}$, for all $n\in\N$.

Combining \cite{dt}*{Theorem 8.2.1} and \eqref{Aleph},
we obtain
\begin{eqnarray*}
\Omega_k^\varphi(F,t)_{w,p} &\le&  c t^k\sum_{0<n\le1/t}n^{k-1}E_n(F)_{w,p} \\
&\le& c t^k N^k E_{k}(f)_{w,p} + c t^k\sum_{N \leq n\le1/t}n^{k-1}E_n(f)_{w,p} \\
& \leq & c(N) t^k  E_{k}(f)_{w,p} + c t^\alpha  , \quad t>0.
\end{eqnarray*}
Hence,
\[
\int_0^1(\Omega_k^\varphi(F,\tau)_{w,p}\,/\tau^{r+1})\,d\tau\le
 \int_0^1 \left( c\tau^{\alpha-r-1} +  c(N) t^{k-r-1}  E_{k}(f)_{w,p} \right)    \,d\tau<\infty,
\]
which, by \cite{dt}*{Theorem 6.3.1(a)}, implies that $F^{(r-1)}$ is
locally absolutely continuous in $(-1,1)$ and
\begin{eqnarray*}
\Omega_{k-r}^\varphi(F^{(r)},t)_{w\varphi^r,p} & \le & c\int_0^t(\Omega_k^\varphi(F,\tau)_{w,p}\,/\tau^{r+1})\,d\tau \\
&\le& c\int_0^t \left( c\tau^{\alpha-r-1} +  c(N) \tau^{k-r-1}  E_{k}(f)_{w,p} \right)\,d\tau \\
& \le & c t^{\alpha-r} + c(N) t^{k-r} E_{k}(f)_{w,p} ,\quad t>0.
\end{eqnarray*}
Finally, taking into account that
\[
\omega_{k-r}^\varphi(F^{(r)},t)_{w\varphi^r,p}=
\omega_{k-r}^\varphi(f^{(r)},t)_{w\varphi^r,p} , \quad t>0 ,
\]
we apply \cite{dt}*{(6.2.9)}  to get
\begin{eqnarray*}
\omega_{k-r}^\varphi(f^{(r)},t)_{w\varphi^r,p} & = & \omega_{k-r}^\varphi(F^{(r)},t)_{w\varphi^r,p} \\
&\le&
c\int_0^t(\Omega_{k-r}^\varphi(F^{(r)},\tau)_{w\varphi^r,p}\,/\tau)\,d\tau \\
&\le& c t^{\alpha-r} + c(N) t^{k-r} E_{k}(f)_{w,p} .
\end{eqnarray*}
This completes the proof.
\end{proof}

Finally, we have the following result analogous to Corollary~\ref{corequivjan} which immediately follows from \ineq{luther} and Theorem~\ref{Inverse}.

\begin{thm}
Let   $wf\in\Lp[-1,1]$, $1\leq p\leq \infty$ and $0\leq r<\alpha<k$.
If
\[
\omega_{k}^\varphi(f,t)_{w,p}\le c t^{\alpha},\quad t>0,
\]
then $f$ is a.e. identical with a function  that has a locally absolutely continuous derivative $f^{(r-1)}$ in $(-1,1)$, and
\[
\omega_{k-r}^\varphi(f^{(r)},t)_{w\varphi^r,p}\le c t^{\alpha-r},\quad t>0.
\]
\end{thm}

\end{document}